\documentclass[a4paper,11pt,oneside]{amsart}

\usepackage[cyr]{aeguill}
\usepackage[english]{babel}
\usepackage{tikz,hyperref}

\usepackage[T1]{fontenc}
\usepackage{amsmath,bbold}
\usepackage{amsfonts}
\usepackage{amssymb}
\usepackage{amsthm,eepic}
\usepackage{mathrsfs}
\usepackage{enumitem}
\usepackage{graphicx}
\usepackage{footnote}
\usepackage{hyperref,tikz}

\allowdisplaybreaks

\usepackage{subcaption}

\renewcommand{\geq}{\geqslant}
\renewcommand{\leq}{\leqslant}
\renewcommand{\epsilon}{\varepsilon}

\hypersetup{colorlinks=true,    linkcolor=MyDarkBlue,
citecolor=BrickRed,       filecolor=BrickRed,       urlcolor=darkgreen
}
\usepackage{color}
\definecolor{darkgreen}{rgb}{0,0.4,0}
\definecolor{MyDarkBlue}{rgb}{0,0.08,0.85}
\definecolor{BrickRed}{rgb}{0.8,0.08,0}

\colorlet{darkgreen}{green!50!black}

\newtheoremstyle{sans}{\parskip}{\parskip}{\itshape}
                       {0pt}{\bfseries\sffamily}{.}{ }{}
\newtheoremstyle{sansplain}{\parskip}{\parskip}{}
                       {0pt}{\bfseries\sffamily}{.}{ }{}

\theoremstyle{sans}
\newtheorem{prop}{Proposition}[section]

\newtheorem{thm}[prop]{Theorem}
\newtheorem{lem}[prop]{Lemma}

\theoremstyle{sansplain}
\newtheorem{rem}[prop]{Remark}

\numberwithin{equation}{section}
\numberwithin{figure}{section}

\usepackage{epsfig}
\usepackage{dsfont}
\def\egaldef{\stackrel{\textnormal{\tiny def}}{=}}

\usepackage{geometry}
\usepackage{mathtools}
\usepackage{eqnarray,cases}

\newcommand\Cc{\mathbb{C}}

\newcommand{\dd}{\mathrm{d}}

\newcommand{\Rr}{\mathbb{R}}

\let\phi=\varphi
\let\epsilon=\varepsilon
\def\DD{\displaystyle}
\def\H{\mathcal{H}}

\def\F{\mathcal{F}}

\renewcommand{\hat}[1]{\widehat{#1}}

\begin{document}

\title[Reflected Brownian motion in a $3/4$-plane]{Stationary Brownian motion in a $3/4$-plane: Reduction to a Riemann-Hilbert problem via Fourier transforms}

\date{\today}

\author{Guy Fayolle}

\address{Inria Paris, 2 rue Simone Iff, CS 42112, 75589 Paris Cedex 12 France \newline\indent
Inria Paris-Saclay, 1 rue Honor\'e d'Estienne d’Orves, 91120 Palaiseau, France} 
\email{guy.fayolle@inria.fr}

\author{Sandro Franceschi}
       
        \address{T\'el\'ecom SudParis, Institut Polytechnique de Paris, 19 place Marguerite Perey, 91120 Palaiseau, France
        } \email{sandro.franceschi@telecom-sudparis.eu}

\author{Kilian Raschel}
       
        \address{Universit\'e d'Angers, Laboratoire Angevin de Recherche en Math\'ematiques, CNRS, SFR MATHSTIC, 2~Boulevard Lavoisier, 49000 Angers, France} \email{raschel@math.cnrs.fr}
        \thanks{This project has received funding from the European Research Council (ERC) under the European Union's Horizon 2020 research and innovation programme under the Grant Agreement No.\ 759702 and from Centre Henri Lebesgue, programme ANR-11-LABX-0020-0.}
        


\begin{abstract}
The stationary reflected Brownian motion in a three-quarter plane has been rarely analyzed in the probabilistic literature, in comparison with the quarter plane analogue model. In this context, our main result is to prove that the stationary distribution can indeed be found by solving a boundary value problem of the same kind as the one encountered  in the quarter plane, up to various dualities and symmetries. The main idea is to start from Fourier (and not Laplace) transforms, allowing to get a functional equation for a single function of two complex variables.
\end{abstract}
\maketitle

\section*{Acknowledgments} We would like to thank the organizers for the invitation to participate in this special issue dedicated to the memory of J.~W.~Cohen.  His important works in queueing theory, combining probabilistic and analytic methods, influenced many researchers around the world in the 1970's. The first author of this paper is particularly indebted to J.~W.~Cohen and O.~Boxma for the great interest they showed in the original study~\cite{FaIa-1979}, some 45 years ago. We further thank the two anonymous referees for very their careful reading and useful remarks.

\section{Introduction} \label{sec:RBM}

\subsection{Reflected Brownian motion in a three-quarter plane}
Since the introduction of reflected Brownian motion in wedges, in the eighties \cite{HaRe-81a,HaRe-81b,VaWi-85,Wi-85}, the probabilistic community has shown a constant interest in this topic. Typical questions deal with the recurrence of the process, the absorption at the vertex of the wedge, the existence of stationary distributions, its computation.  A related problem was solved in \cite{BaFa-87}, for a diffusion process in the positive quarter plane. For more details, we refer to the introductions of \cite{FrRa-19} and \cite{FaFrRa-21}. A precise semimartingale definition of reflected Brownian motion will be given in Section~\ref{sec:model}. In this work, we will restrict ourselves to the stationary case.

Generally speaking, an obliquely reflected Brownian motion in a two-dimensional wedge of opening $\zeta\in (0, 2\pi)$ is characterized by its drift $\mu\in\mathbb R^2$ and two reflection angles $(\delta,\varepsilon)\in(0, \pi)^2$, see Figures \ref{fig:intro_scaling} and \ref{fig:three_quarter_plane} for examples. The covariance matrix is taken to be the identity. A suitable linear transform allows to reduce the whole range of parameter angles $\zeta\in (0, 2\pi)$ to three cases: the quarter plane ($\zeta\in (0,\pi)$), the three-quarter plane (when $\zeta\in (\pi,2\pi)$) and the limiting half-plane case $\zeta=\pi$. Doing so, the covariance matrix is no longer the identity matrix but has a general form \eqref{eq:covariance_matrix}.

While the early articles \cite{VaWi-85,Wi-85} most dealt with the general case $\zeta\in (0, 2\pi)$, the case of convex cones $\zeta\in (0, \pi]$ has attracted much more attention \cite{HaRe-81a,HaRe-81b,BaFa-87,DaMi-11,FrRa-19,BoElFrHaRa-21}. However, as explained in the introduction of \cite{FaFrRa-21}, there are numerous reasons to look at the non-convex, three-quarter plane situation. Our particular motivation is provided by the discrete framework of random walks (or queueing networks). Indeed, in the same way as in the quarter plane, reflected Brownian motion has been introduced to study scaling limits of large queueing networks (see Figure~\ref{fig:intro_scaling}), a Brownian model in a non-convex cone could approximate discrete random walks on a wedge having obtuse angle (see Figure~\ref{fig:intro_scaling} for an example). Such random walks have an intrinsic interest and have already been studied in a number of cases: see \cite{BM-16,RaTr-19,EP-22,Bo-22} in the combinatorial literature and \cite{Tr-22,Mu-19} for more probabilistic works. 

\begin{figure}[hbtp]
\centering
\includegraphics[scale=0.5]{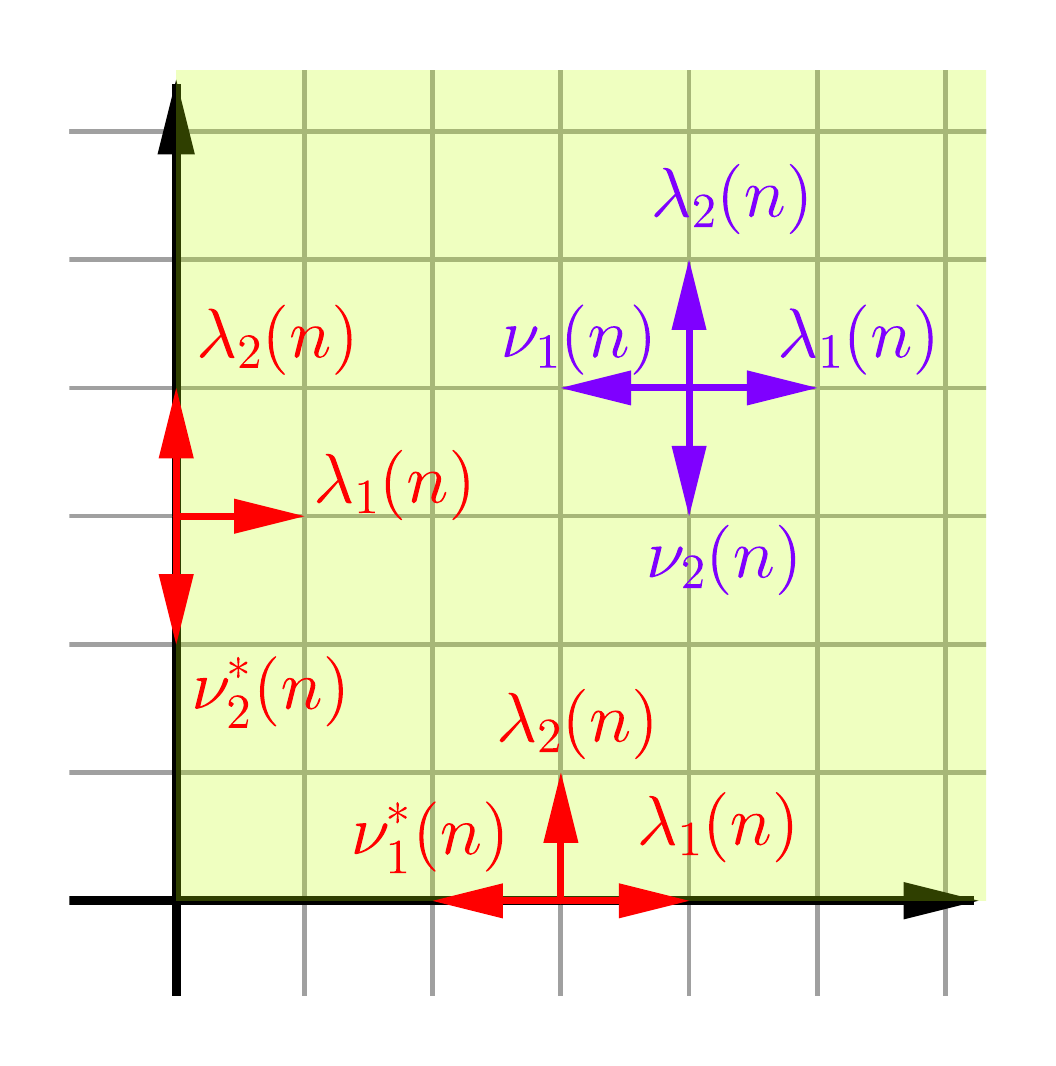}
\includegraphics[scale=0.5]{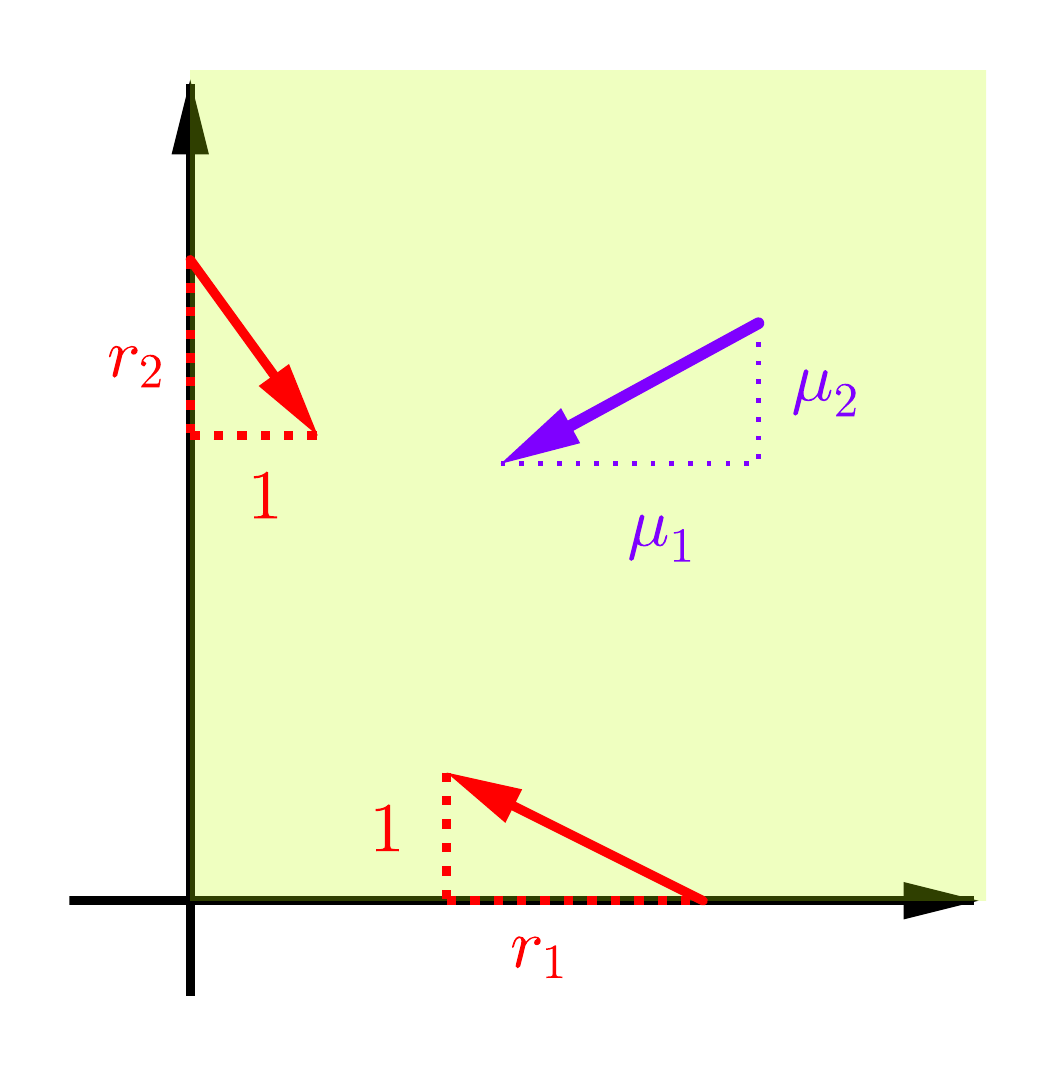}

\includegraphics[scale=0.5]{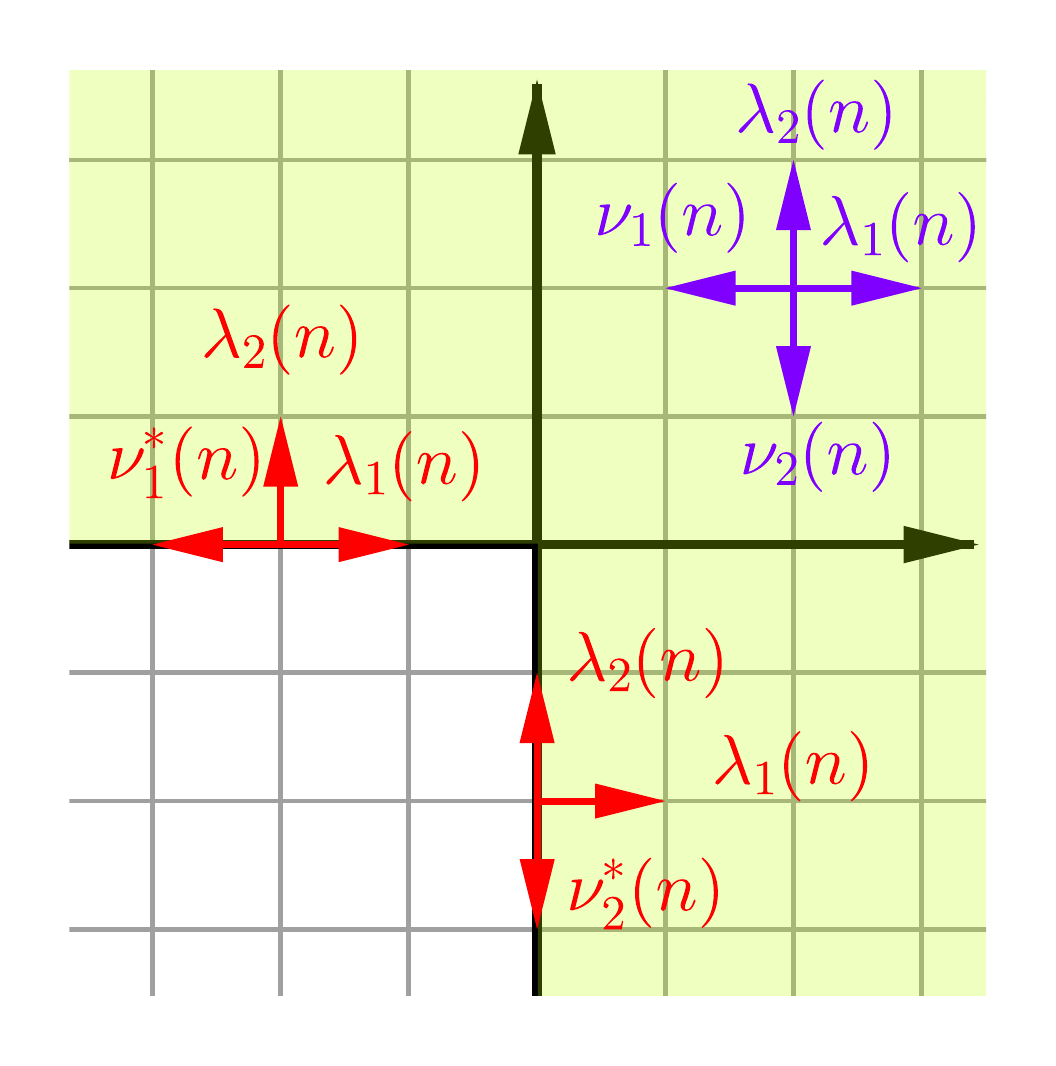}
\includegraphics[scale=0.5]{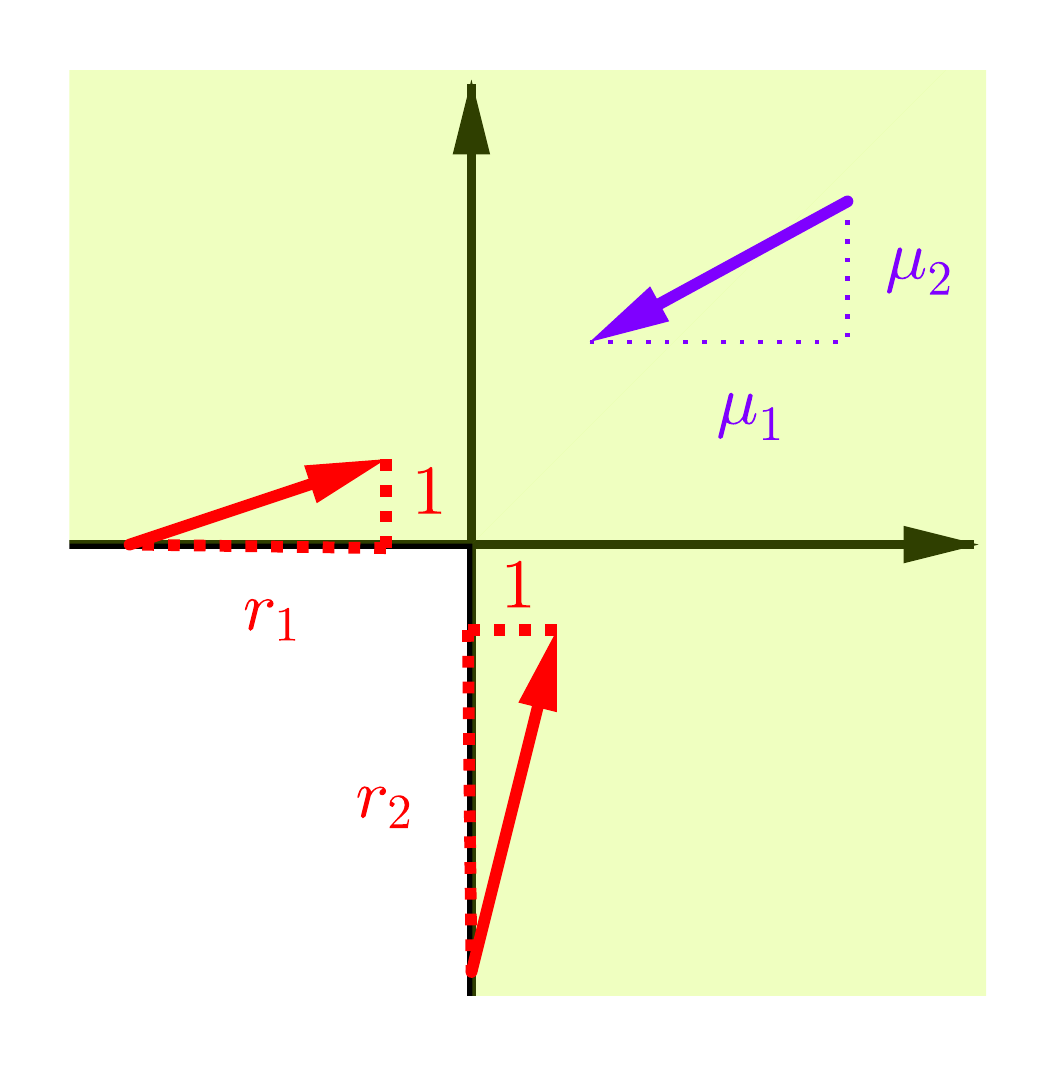}
\caption{Scaling limit of some queueing systems towards reflected Brownian motion. Top left picture: transition rates of a random walk (two coupled processors). Taking $\lambda_{i}(n),\nu_i(n)\to\frac{1}{2}$, $\sqrt{n}(\lambda_i(n)-\nu_i(n))\to \mu_i$ and $\nu_i^*(n)\to\frac{r_i+1}{2}$, the discrete process converges to the reflected Brownian motion with parameters described as on the top right picture (with identity covariance matrix). See \cite{Re-84} for the original proof. For the exact same reasons, in the three-quarter plane, the discrete model on the bottom left picture converges to the reflected Brownian motion on the bottom right display.}
\label{fig:intro_scaling}
\end{figure}

\subsection{Main results and scheme of the proofs}

The present work is a companion paper of \cite{FaFrRa-21}, where we proved that the three-quarter plane stationary distribution could be found by solving a two-dimensional \textit{vector} boundary value problem (BVP) for the associated Laplace transforms. Here we go much further, by showing that one can actually reduce the latter to a classical \textit{scalar} BVP. This has two main consequences, which we shall present in our paper. First, we will obtain explicit contour integral expressions for the Laplace transforms, see Theorem~\ref{thm:main_result} for the precise statement. Second, we will accurately compare the quarter plane and three-quarter plane cases, and understand the transformations and symmetries allowing to pass from one representation to the other.

In order to explain the peculiarities and difficulties intrinsic to the non-convex setting, let us recall the general approach when solving a problem related to (random walks or) Brownian motion in the quarter plane. There are two main steps: 
\begin{enumerate}[label=(\roman{*}),ref=(\roman{*})]
   \item\label{it:eq}stating a functional equation for the Laplace transform of the stationary measure;
   \item\label{it:BVP} deducing a BVP  from the main functional equation, and finally solving this BVP by means of contour integral representations.
\end{enumerate}

In the case of convex domains, step \ref{it:eq} is now routine, see for instance \cite{BaFa-87,FrRa-19}, and \cite{CoBo-83,Co-92,FIM-2017} in the discrete setting. We start from It\^o-Tanaka formula, then prove a so-called ``basic adjoint relationship'', which, applied to exponential test functions, leads to a functional equation for the bivariate Laplace transform
\begin{equation}
\label{eq:Laplace_transform_generic}
   L(p,q) = \int e^{p z_1+ q z_2} \dd\Pi(z_1,z_2),
\end{equation}
$\Pi$ denoting the invariant measure we are looking for. When the integration domain in \eqref{eq:Laplace_transform_generic} is the positive quarter plane, the Laplace transform is clearly analytic (at least) when $\Re(p)\leq 0$ and $\Re(q)\leq 0$. This analyticity property is crucial for step~\ref{it:BVP}.

In the case of the three-quarter plane, giving a sense to the Laplace transform \eqref{eq:Laplace_transform_generic}, and a fortiori deriving a functional equation, turns out to be non-trivial. Indeed, the exponential function in the integrand of \eqref{eq:Laplace_transform_generic} is bounded only in a half-plane  and unbounded in its complement. Therefore, the proper definition and convergence domain of \eqref{eq:Laplace_transform_generic} require fine estimates.

A first step (which explains the title of the present study) is to use Fourier transform, i.e., purely imaginary values of $(p,q)$ in \eqref{eq:Laplace_transform_generic}, in order to take advantage of the integrability of the stationary measure $\Pi$. This is however not sufficient, as we  need to continue meromorphically the Laplace transform onto an open domain of $\mathbb C^2$ containing $(0,0)$. For that purpose, we introduce a continuation procedure related to the one used  in \cite{FaFrRa-21}, and close to those in \cite{FIM-2017,EP-22} in the discrete setting. This is done in Section~\ref{sec:resolution} and Section~\ref{sec:continuation} in the appendix; the main idea is to split the three-quarter plane into two convex cones, and to state two intermediate functional equations for Laplace transforms whose convergence is clear. The  main functional equation is obtained in Section~\ref{sec:functional_equation}, see in particular Proposition~\ref{prop:main_eqfunc}.

Once the functional equation is established, the techniques we employ to derive and solve a BVP (step \ref{it:BVP} above) are more classical, and may be adapted for instance from \cite{BaFa-87,FrRa-19}. We  obtain explicit formulas in Section~\ref{sec:resolution}, see in particular  the result given 
in~Theorem~\ref{thm:main_result}. 

In this work, we also  focus on the main similarities and differences in the formulas for the Laplace transforms in the quarter plane and three-quarter plane, more than on the explicit formulas themselves. This is the topic of Section~\ref{sec:diff}. To select one example, in the quarter plane (resp.\ three-quarter plane), the Laplace transforms satisfy boundary value problems on the left (resp.\ right) branch of the same hyperbola, with an analytic behavior on the left (resp.\ right), see Figure~\ref{fig:domain}. This illustrates that several dualities and symmetries relate the two settings.

The literature contains a few examples of discrete walks in a three-quarter plane, which, starting from Fourier transforms, might also be solved by means of a single functional equation, through subtle manipulations on the various generating functions (see e.g., \cite{BM-16,RaTr-19,Tr-22,EP-22,Bo-22}).

\section{Semimartingale reflected Brownian motion in three-quarter plane}
\label{sec:model}

\subsection{Definition of the process}

We denote the three-quarter plane as
\begin{equation*}
   S \egaldef \{ (z_1,z_2)\in \mathbb{R}^2 : z_1 \geq 0 \text{ or } z_2 \geq 0 \}.
\end{equation*}
The parameters of the model are the drift $\mu=(\mu_1,\mu_2)$, the reflection vectors $R_1=(r_1,1)$ and $R_2=(1,r_2)$, and the covariance matrix
\begin{equation}
\label{eq:covariance_matrix}
\Sigma= \left(
\begin{array}{cc}
\sigma_1 & \rho \\ 
\rho & \sigma_2
\end{array} 
\right),
\end{equation}
see Figure~\ref{fig:three_quarter_plane}. Throughout this study, $\Sigma$ will be assumed to be elliptic, that is
\begin{equation}
\label{eq:elliptic}
   \sigma_1\sigma_2-\rho^2>0,
\end{equation}
and we shall not consider degenerate cases (hypoellipticity, $\Sigma$ non-definite positive, etc.), although they admit an interesting behavior, as shown for example in \cite{IcKa-22}. 
The study of the degenerate case $\sigma_1\sigma_2-\rho^2=0$ would very likely lead to the study of boundary value problems on parabolas, and no longer on hyperbolas as it is the case here (see in this respect Remark~\ref{rem:para}).

\begin{figure}[hbtp]
\centering
\includegraphics[scale=0.4]{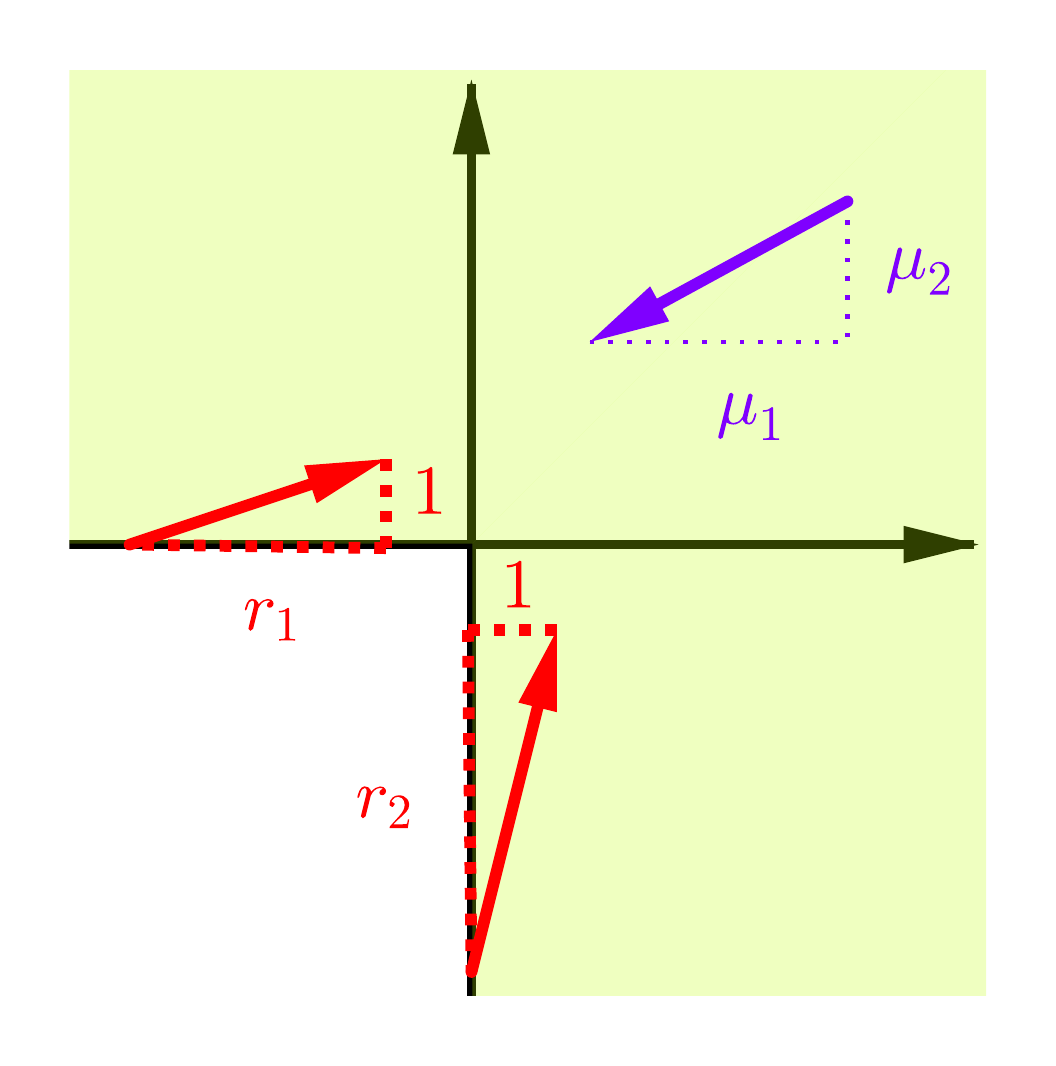}
\caption{In green color, the three-quarter plane $S$, in blue the drift $\mu$ and in red the reflection vectors $R_1$ and $R_2$.}
\label{fig:three_quarter_plane}
\end{figure}

More specifically, we define the obliquely reflected Brownian motion $Z_t=(Z_t^1,Z_t^2)$ in the three-quarter plane $S$ as follows:
\begin{equation*}
\begin{cases}
Z_t^1\egaldef Z_0^1+ W_t^1+ \mu_1 t+ r_1 L_t^1 + L_t^2,
\\
Z_t^2\egaldef Z_0^2+ W_t^2+ \mu_2 t+ L_t^1 + r_2 L_t^2,
\end{cases}
\end{equation*}
where $W_t$ is a planar Brownian motion of covariance $\Sigma$, $L^1_t$ is (up to a constant) the local time on the negative part of the abscissa ($z_1\leq 0$) and $L_t^2$ is the local time on the negative part of the ordinate axis ($z_2\leq 0$). In case of a zero drift, such a semimartingale definition of reflected Brownian motion is proposed in the reference paper \cite{Wi-85}; it readily extends to our drifted case.

Throughout this paper, we assume that the process is positive recurrent and has a unique stationary distribution (or invariant measure). As we shall see, this is equivalent to
\begin{equation}
\label{eq:CNS1_ergodic}
   \mu_1<0 \quad \text{and} \quad \mu_2<0,
\end{equation}
together with
\begin{equation}
\label{eq:CNS2_ergodic}
   \mu_1-r_1\mu_2>0 \quad \text{and} \quad
   \mu_2-r_2\mu_1>0,
\end{equation}
keeping in mind that $r_1$ and $r_2$ are positive.

Under conditions \eqref{eq:CNS1_ergodic} and \eqref{eq:CNS2_ergodic}, we will denote by $\Pi$ this probability measure and by $\pi$ its density. We also define the boundary invariant measures by
\begin{equation*}
   {\nu}_{1} (A) = \mathbb{E}_\Pi \int_0^1 \mathrm{1}_{A\times\{0\} } (Z_s) \mathrm{d}L_s^1
   \quad \text{and} \quad
   {\nu}_{2} (A) = \mathbb{E}_\Pi \int_0^1 \mathrm{1}_{\{0\} \times A} (Z_s) \mathrm{d}L_s^2.
\end{equation*}
The measure ${\nu}_{1}$ has its support on $\{z_1 \leq 0 \}$ and ${\nu}_{2}$ has its support on $\{z_2 \leq 0 \}$. We will also denote by $\nu_1 (z_1)$ and  $\nu_2 (z_2)$ their respective densities.

\subsection{Basic adjoint relationship}
Our approach is based on the following identity, called basic adjoint relationship, which in the orthant case is proved for instance in \cite{DaHa-92,HaWi-87}.

\medskip

\begin{prop}
\label{prop:BAR}
For any function $f:\Rr^2\to\Cc$ of class $\mathcal{C}^2$, assuming the integrals below converge, we have
\begin{equation*}
\int_S \mathcal{G} f(z_1,z_2) \dd \Pi(z_1,z_2) +  \int_{-\infty}^0 R_1 \cdot \nabla f(z_1,0) \dd \nu_1 (z_1) + \int_{-\infty}^0 R_2 \cdot \nabla f(0,z_2)\dd \nu_2 (z_2)=0,
\end{equation*}
where the generator is equal to
\begin{equation*}
\mathcal{G} f = \frac{1}{2} \left(\sigma_1 \frac{\partial^2 f}{\partial z_1^2}+2\rho \frac{\partial^2 f}{\partial z_1 \partial z_2} +\sigma_2 \frac{\partial^2 f}{\partial z_2^2}
 \right)
+ \mu_1 \frac{\partial f}{\partial z_1}
+\mu_2 \frac{\partial f}{\partial z_2}.
\end{equation*}
\end{prop}
\begin{proof}
We apply the well-known It\^o-Tanaka formula to the semimartingale $Z_t$, see Theorem~1.5 in \cite[Chap.~VI \S 1]{Revuz1999}. We obtain
\begin{equation}
\label{eq:ito}
   f(Z_t)=f(Z_0)+\int_0^t \mathcal{G}f (Z_s) \mathrm{d}s +\int_0^t \nabla f (Z_s) \cdot \dd W_s +  \sum_{i\in\{1,2\}} \int_0^t R_i \cdot \nabla f (Z_s) \dd L_s^i .
\end{equation}
To conclude, it suffices to take the expectation over $\Pi$ in the above equality.
\end{proof}
%
\section{The main functional equation}
\label{sec:functional_equation}


Our goal is to apply the basic adjoint relationship of Proposition~\ref{prop:BAR}, in order to construct a functional equation for the two-dimensional Fourier transform of~$\Pi$. 
\begin{itemize}
\item In the case of a convex cone, after a linear change of variables in the $(z_1,z_2)$-plane, it suffices to choose $f(z_1,z_2)=e^{p z_1+ q z_2}$, with $\{\Re{(p)}\geq0,\,\Re{(q)}\geq0\}$, see \cite{DaMi-11,FrRa-19}, and a functional equation is obtained for  the Laplace transform of $\Pi$, say $F(p,q)$.
\item However, when the cone is not convex, some integral transforms will not converge, and one has to divide the three-quarter plane into  two convergence regions. This leads to a system of two functional equations for the Laplace transforms, see \cite{FaFrRa-21}, which is more awkward to solve.
\end{itemize}

We show hereafter  that, \emph{by starting from  the Fourier transform of the stationary distribution},  it is  still  possible to find a single kernel equation which can be solved by reduction to a BVP for a function of a single complex variable.

%
%

The three following Fourier transforms (related to the invariant measure $\Pi$ in $S$)
\begin{equation}
\label{eq:fourier}
\begin{cases}
\DD L(p,q) \egaldef \int_{S} e^{p z_1+ q z_2} \dd\Pi(z_1,z_2),  \\[0.4cm]
\DD A(p) \egaldef \int_{-\infty}^0 e^{p z_1} \dd\nu_1 (z_1), \\[0.4cm]
\DD B(q) \egaldef \int_{-\infty}^0 e^{q z_2} \dd\nu_2 (z_2),
\end{cases}
\end{equation}
are a priori well defined in the domain $\{\Re(p)=0,\,\Re(q)=0\}$.


\medskip

\begin{prop}
\label{prop:main_eqfunc} 
For all $(p,q)$ in  the region $\{\Re(p)=\Re(q)=0\}$, we have
\begin{equation}
\label{eq:eqfunc}
  \boxed{K(p,q)L(p,q) + u(p,q) A(p) + v(p,q) B(q) =0}
\end{equation}
where the kernel $K(p,q)$ is given by
\begin{equation} 
\label{eq:kernel}
 K(p,q) \egaldef \frac{1}{2} \left(\sigma_1 p^2 +2\rho pq +\sigma_2 q^2\right) 
 + \mu_1 p+\mu_2q,
\end{equation}
while $u(p,q)$ and $v(p,q)$ are the linear functions 
\begin{equation}
\label{eq:coef}
u(p,q) \egaldef r_1p + q, \quad v(p,q) \egaldef r_2q + p.
\end{equation}
\end{prop}
\begin{proof}
It is a direct application of Proposition~\ref{prop:BAR}, by choosing  $f(z_1,z_2)=e^{p z_1+ q z_2}$, where $(p,q)$ are purely imaginary complex numbers, so that the integrals in~\eqref{eq:ito} are well defined. Further details are omitted.
\end{proof}
Equation \eqref{eq:eqfunc} says that the probability measure in $S$ is given in terms of the probability measures on the boundary formed by the two negative half-axes, which corresponds to the intuition.

In this respect, the following relations (mass conservation), which are immediately obtained from \eqref{eq:kernel}, are interesting. Observe that $L(0,0)=\Pi (S)=1$, $A(0)=\nu_1(\mathbb{R}_-)$ and $B(0)=\nu_2(\mathbb{R}_-)$.
\begin{equation*}
\left\{ \begin{array}{l}
\mu_2 L(0,0) + A(0) +r_2B(0)=0, \\[0.5cm]
\mu_1 L(0,0) + r_1 A(0) + B(0)=0.
\end{array} \right. \qquad 
\left\{ \begin{array}{l} 
\DD A(0)=\frac{\mu_1r_2-\mu_2}{1-r_1r_2} , \\[0.5cm]
\DD B(0)=\frac{\mu_2r_1-\mu_1}{1-r_1r_2} .
\end{array}\right.
\end{equation*}
Note that the apparently peculiar case  $1-r_1r_2=0$ has not to be considered, since if $r_1r_2=1$ the two conditions in~\eqref{eq:CNS2_ergodic} cannot be simultaneously satisfied, in which case the process is not positive recurrent.

In the next lemmas, we make a suitable  analytic continuation with respect to $(p,q)$ of the functions defined in~\eqref{eq:eqfunc}.
\smallskip
\begin{lem}
\label{prop:anacont1}
\mbox{}
\begin{enumerate}[label={\it(\arabic{*})},ref={\it(\arabic{*})}]
\item\label{it:analy-1}The functions $A(p)$ and $B(q)$ are analytic in their respective domains $\{\Re{(p)}\geq0\}$ and $\{\Re{(q)}\geq0\}$.
\item\label{it:analy-2}The function $L(p,q)$ is analytic with respect to $(p,q)$ in the domain
\begin{equation}
\label{eq:domain}
   \{\Re{(p)}\geq0\} \cap \{\Re{(q)}\geq0\} \cap \{K(p,q)\neq0\}.
\end{equation}
\end{enumerate}
\end{lem}

\begin{proof} 
The point \ref{it:analy-1} is immediate  from the definition given in~\eqref{eq:fourier}. 
As for point \ref{it:analy-2}, we use \eqref{eq:eqfunc} to make the analytic continuation of $L(p,q)$ in each variable $p$ and  $q$, and we apply Hartogs' and Osgood's theorems (see, e.g., \cite{BoMa-48}) to deduce the analyticity in both variables $(p,q)$ in the region~\eqref{eq:domain}.
\end{proof}

The algebraic curve $\{K=0\}$ has genus $0$, and its branches over the $q$-plane (resp.\ $p$-plane) will be denoted by $P_i(q)$ (resp.\ $Q_i(p))$, for $i=1,2$. By definition, they satisfy the equations
\begin{equation}
\label{eq:Kroot}
  \boxed{K(P_i(q),q)=0 \quad\text{and}\quad K(p, Q_i(p)) = 0.}
\end{equation}
Introduce the hyperbola $\H_p$ with equation
\begin{equation}
\label{eq:HP}
   (\rho^2 -\sigma_1\sigma_2)x^2 + \rho^2 y^2 
  +2(\rho\mu_2-\mu_1\sigma_2)x + \frac{\mu_1(2\rho\mu_2-\sigma_2\mu_1)}{\sigma_1} = 0.
\end{equation}

\begin{rem} \label{rem:para} When $\rho^2 -\sigma_1\sigma_2=0$, equation~\eqref{eq:HP} represents in fact a parabola. As mentioned in Section~\ref{sec:model}, this  limiting case will not be considered in the sequel. 
\end{rem}

\begin{lem}
\label{lem:Pq}
The functions $P_i(q)$, $i=1,2$, are analytic in the whole complex plane cut along $(-\infty,q_1]\cup [q_2,\infty)$, where the branch points $q_1<0$ and $q_2>0$ are the two real roots of the equation 
\begin{equation}
\label{eq:bpP}
   (\rho^2-\sigma_1\sigma_2)q^2 + 2(\rho\mu_1 - \sigma_1 \mu_2)q + \mu_1^2 = 0.
\end{equation}
The branches $P_1(q)$ and $P_2(q)$ satisfy the following properties:
\begin{enumerate}[label={\it(\arabic{*})},ref={\it(\arabic{*})}]
\item\label{it:sep}They are separated, in the sense that
\begin{equation}
\label{eq:sepP}
\begin{cases}
\DD \Re(P_1(q))  \leq \Re(P_2(q)), \quad \forall q \in \Cc,  \\[0.2cm]
\DD P_1(0) = 0, \quad P_2(0) = \frac{-2\mu_1}{\sigma_1}>0.
\end{cases}
\end{equation}
\item\label{it:hyp}For $\rho<0$, they map the cut $(-\infty,q_1]$ (resp.\ $[q_2,\infty)$) 
onto  the left component of $\H_p$, denoted by $\H_p^-$  (resp.\ right component of $\H_p$, denoted by $\H_p^+$).
\item For $\rho>0$, just exchange  the left and right components in the previous statement.
\item\label{it:deg}For $\rho=0$, the hyperbola is degenerate and its two components coincide with the vertical line of abscissa $x=-\mu_1/\sigma_1$.
\end{enumerate}
\end{lem}
To simplify the notation, we will note the branch of hyperbola
\begin{equation*}
\H \egaldef
\begin{cases}
\H_p^+ & \text{if } \rho<0,
\\
\H_p^- & \text{if } \rho>0.
\end{cases}
\end{equation*}
See Figure~\ref{fig:domain} for a representation of the branch $\H$ of the hyperbola according to the sign of $\rho$.
Symmetrically, we define the functions $Q_i(p)$, the branch points $p_1<0$ and $p_2>0$ and the hyperbola $\H_q$.

\begin{proof}
The fact that the roots of \eqref{eq:bpP} are real is immediate since $\rho^2-\sigma_1\sigma_2<0$. The branch points of the algebraic function $P(q)$ are the zeros of the discriminant of $K(p,q)=0$ viewed as a polynomial in $p$, so that \eqref{eq:bpP} follows directly. 

We now prove~\ref{it:sep}. Let in \eqref{eq:Kroot} $q=i\gamma$ and $P(q)\egaldef \alpha+i\beta$, with real $\alpha,\beta$. After separating  real and imaginary parts, we obtain 
\begin{equation}
\label{eq:sepP2}
\begin{cases}
\DD \frac{\sigma_1}{2}\alpha^2 + \mu_1\alpha -\left(\frac{\sigma_2}{2}\gamma^2 + \rho\beta\gamma + \frac{\sigma_1}{2} \beta^2\right) = 0, \\[0.3cm] 
\DD \beta(\sigma_1\alpha +\rho\gamma) + \gamma(\rho\alpha +\mu_2) = 0.
\end{cases}
\end{equation}
One checks that the first equation in \eqref{eq:sepP2}, viewed as a polynomial in~$\alpha$, has two real roots with opposite signs, since the quadratic polynomial in~$\gamma$ is always positive, due to the ellipticity condition~\eqref{eq:elliptic}. Accordingly, we will denote by $P_1(q),P_2(q)$ the branches satisfying 
\begin{equation*}
   \Re(P_1(ix)) \leq 0 \leq \Re(P_2(ix)),  \quad  \forall x\in \mathbb{R}.
\end{equation*}
Then the first property in \eqref{eq:sepP} is a direct application of the maximum modulus principle to the function $\exp(P_1(q)-P_2(q))$ in the domain 
$\mathbb C\setminus \{(-\infty,q_1]\cup [q_2,\infty)\}$,
keeping in mind that on the  cuts $(-\infty,q_1]$ and $[q_2,\infty)$, the branches $P_1(q)$ and $P_2(q)$ are complex conjugate, in which case $\vert\exp(P_1(q)-P_2(q))\vert=1.$

As for point \ref{it:hyp}, the analytic expression \eqref{eq:HP} of the hyperbola~$\H_p$ follows from direct computations. On the other hand, by \eqref{eq:kernel} and \eqref{eq:Kroot}, we have
   \begin{equation*}
P_1(q_i) + P_2(q_i) = \frac{-2(\rho q_i +\mu_1)}{\sigma_1}, \quad  i=1,2.
\end{equation*}
When $\rho<0$, Equation~\eqref{eq:HP} shows that the hyperbola $\H$ crosses the $x$-axis at two points with positive abscissas, whence 
\begin{equation*}
\frac{-2(\rho q_2 +\mu_1)}{\sigma_1}=P_1(q_2) + P_2(q_2)>P_1(q_1) + P_2(q_1)>0.
\end{equation*}
Point \ref{it:deg} is elementary after putting $\rho=y=0$ in~\eqref{eq:HP}, and the proof of the lemma is concluded.
\end{proof}

\begin{figure}[hbtp]
\centering
   \begin{subfigure}{\textwidth}
         \centering
         \includegraphics[trim = 0cm 1cm 0cm 1cm, clip,scale=1.3]{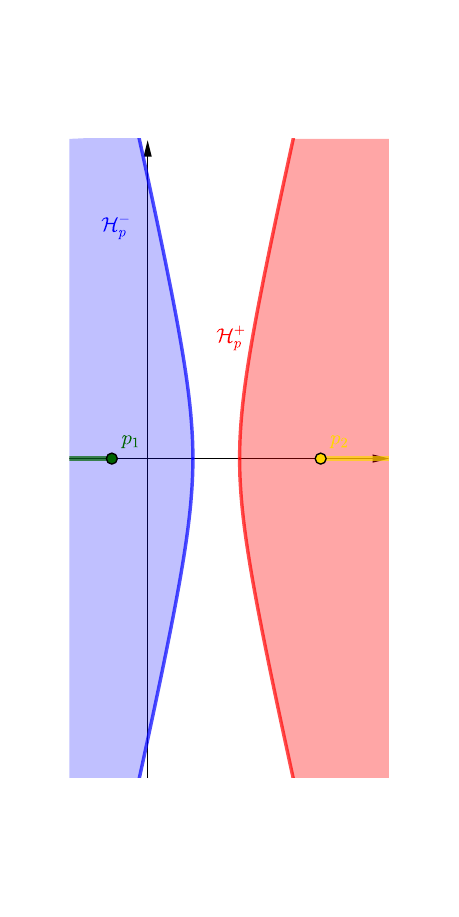}
\includegraphics[trim = 0cm 1cm 0cm 1cm, clip,scale=1.3]{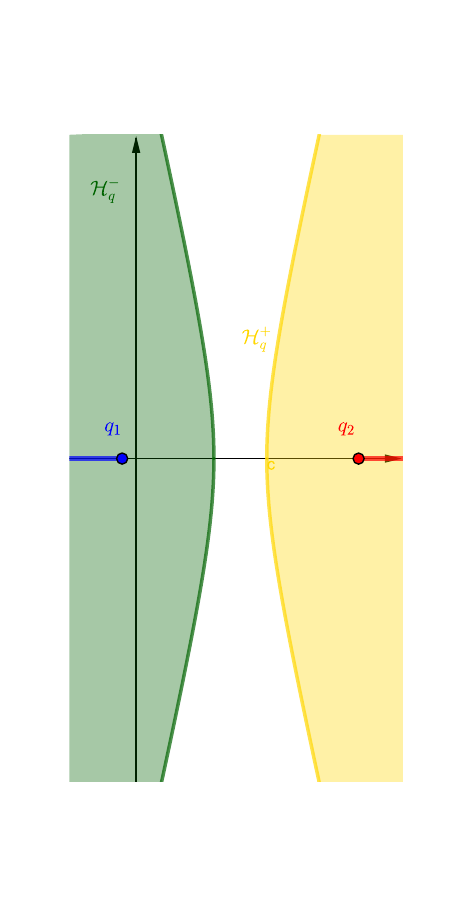}
         \caption{Case $\rho<0$}
     \end{subfigure}
     \hfill
      \begin{subfigure}{\textwidth}
               \centering
     \includegraphics[trim = 0cm 1cm 0cm 1cm, clip,scale=1.3]{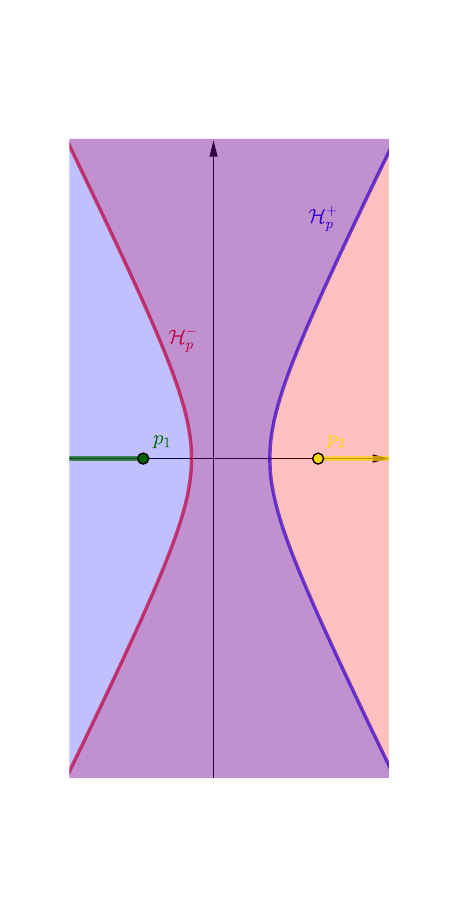}
\includegraphics[trim = 0cm 1cm 0cm 1cm, clip,scale=1.3]{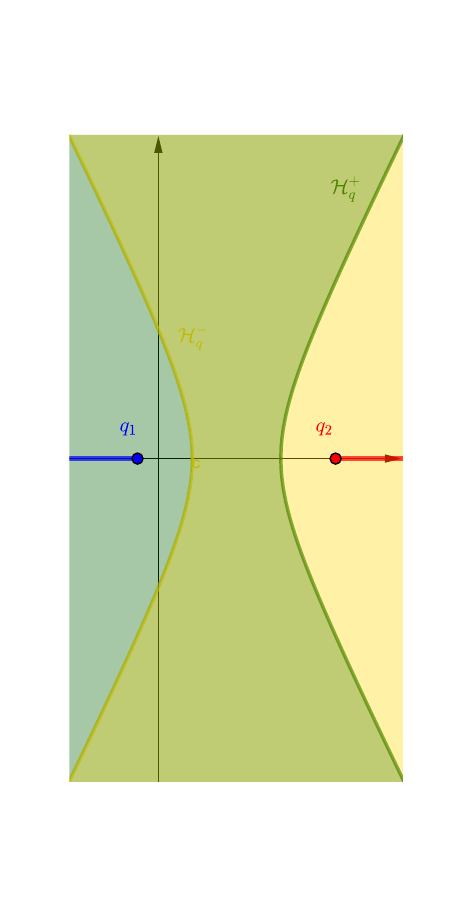}
         \caption{Case $\rho>0$}
     \end{subfigure}
\caption{On the left (resp.\ right), the complex plane of the variable $p$ (resp.\ $q$). The cut $[q_2,\infty)$, represented by the red half-line, is mapped by $P_1$ and $P_2$ onto the red branch of the hyperbola $\mathcal{H}_p^+$ when $\rho<0$, and $\mathcal{H}_p^-$ when $\rho>0$. The same holds for other colors. 
The colored domains are the corresponding domains of the respective BVPs. The domains corresponding to the quarter plane (resp.\ three-quarter plane) BVP are represented in blue and green (resp.\ red and yellow).}
 \label{fig:domain}
\end{figure}

\begin{figure}[hbtp]
\centering
   \begin{subfigure}{\textwidth}
         \centering
         \includegraphics[trim = 0.5cm 0.5cm 0.5cm 0.5cm, clip,scale=1.3]{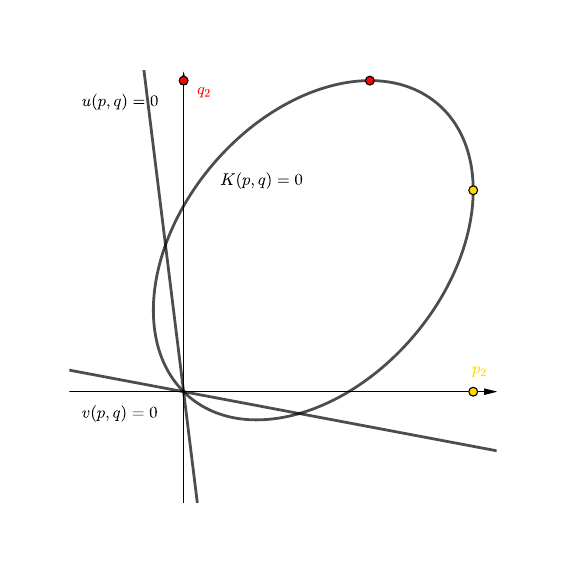}
\includegraphics[trim = 0.5cm 0.5cm 0.5cm 0.5cm, clip,scale=1.3]{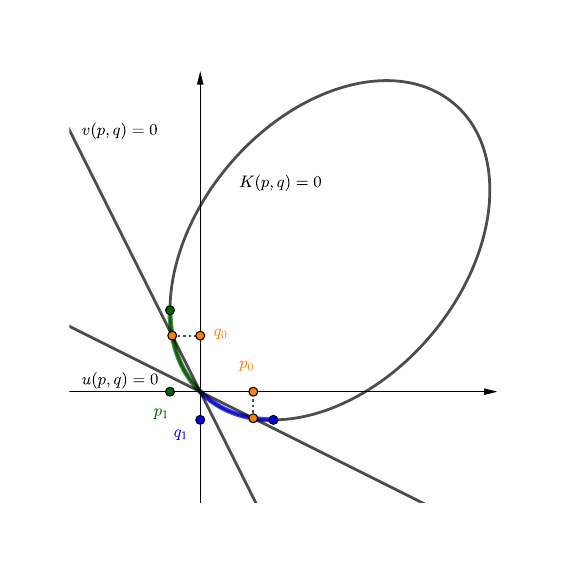}
         \caption{Case $\rho<0$}
     \end{subfigure}
     \hfill
      \begin{subfigure}{\textwidth}
               \centering
     \includegraphics[trim = 0.5cm 0.5cm 0.5cm 0.5cm, clip,scale=1.3]{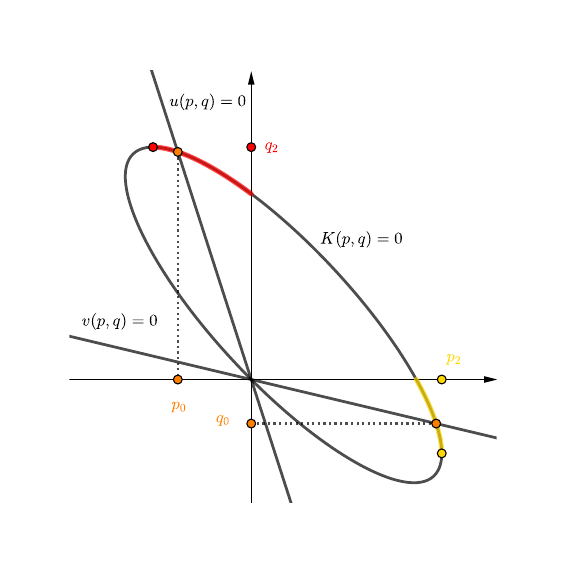}
\includegraphics[trim = 0.5cm 0.5cm 0.5cm 0.5cm, clip,scale=1.3]{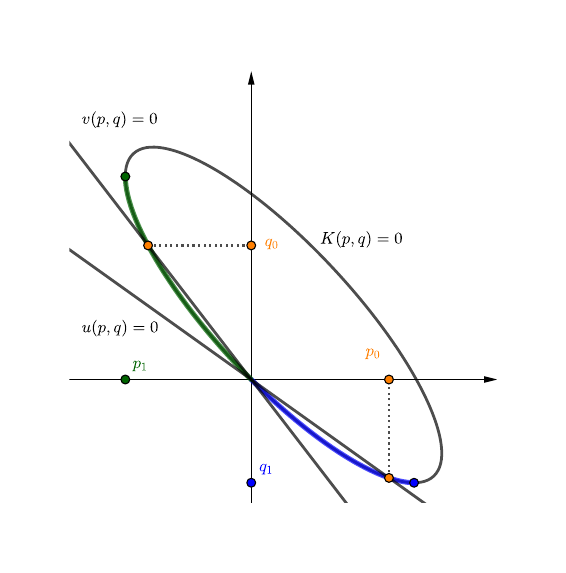}
         \caption{Case $\rho>0$}
     \end{subfigure}
\caption{Representation of the ellipse $\{(p,q)\in\mathbb R^2:K(p,q)=0\}$, with $p$ in abscissa and $q$ in ordinate. The case of the $3/4$ (resp.\ $1/4$) plane is displayed on the left (resp.\ right). The pole $p_0$ of $A(p)$ (resp.\ the pole $q_0$ of $B(q)$) is in orange. This pole exists if and only if the line $u(p,q)=0$ (resp.\ $v(p,q)=0$) intersects the corresponding colored arc of the ellipse. The colors match with Figure~\ref{fig:domain}. Observe that, when $\rho<0$, the functions $A(p)$ and $B(q)$ for the $3/4$-plane cannot have a pole in the domain of the BVP, which in this case is included in the right half-plane.}
 \label{fig:ellipse}
\end{figure}

\section{Solving the boundary value problem}
\label{sec:resolution}

Letting $\H_{+}^{o}$ denote the domain inside the right component of the hyperbola $\H$, the following theorem holds.
\medskip
\begin{thm}
 The function $A(p)$ in \eqref{eq:fourier}, sought to be meromorphic in the domain $\H_{+}^{o}$, satisfies the  BVP
\begin{equation*}
f(P_1(q),q)A(P_1(q)) - f(P_2(q),q)A(P_2(q)) =0, \quad \forall q\in[q_2,\infty),
\end{equation*}
which can be rewritten in the equivalent form
\begin{equation}
\label{eq:statement_BVP}
 \boxed{g(p)A(p) - g(\bar{p})A(\bar{p}) = 0, \quad \forall p \in  \H,}
\end{equation}
where
\begin{equation*}
f(p,q) = \frac{u(p,q)}{v(p,q)}, \qquad g(p)=f(p,Q_2(p)).
\end{equation*}
\label{thm:BVP}
\end{thm}

\begin{proof}

By Lemma~\ref{prop:anacont1}, we know that $L(p,q)$  is analytic in the $(p,q)$-domain
\begin{equation*}
  \F \egaldef \{\Re(p)\ge0\} \cap \{\Re(q)\ge0\} \cap \{K(p,q)\ne0\}.
\end{equation*}
Let us express  $S$ as the disjoint union of the two following convex wedges
\begin{equation*}
S_1\egaldef\{z_1\leq z_2 \text{ and } z_2 \geq 0 \}, \quad S_2\egaldef S\setminus S_1.
\end{equation*}
Then, for all $(p,q)$ satisfying $\{\Re(p)=\Re(q)=0\}$, we can write 
\begin{equation*}
L(p,q) = L_1(p,q) + L_2(p,q),
\end{equation*}
where 
\begin{equation*} 
L_1(p,q) = \int_{S_1} e^{p z_1+ q z_2} \dd\Pi(z_1,z_2 ) \quad \text{and} \quad  
L_2(p,q) = \int_{S_2} e^{p z_1+ q z_2} \dd\Pi(z_1,z_2 ). 
\end{equation*}
Letting  $D\egaldef \{(p,q):\Re(p+q)\leq0\}$, $D_1 \egaldef \{\Re(p)\ge0\}$ and $D_2 \egaldef \{\Re(q)\ge0\}$, it is immediate to check that $L_1(p,q)$ and $L_2(p,q)$ have an analytic continuation in the respective domains 
\begin{equation*}
D_1  \cap D  \quad \text{and}\quad  
D_2  \cap D .
\end{equation*}
On the other hand, we proved in a previous work (see (4.1) and (4.2) in \cite{FaFrRa-21}) the relations
\begin{align}
 F(p,q) &= K(p,q)L_1(p,q)+u(p,q)A(p) , \label{eq:eqL1} \\
F(p,q) &= -K(p,q)L_2(p,q) - v(p,q)B(q), \label{eq:eqL2}
\end{align}
where $F(p,q)$ is meromorphic in the region $D$ introduced above.

We are now entitled to  make the meromorphic continuation of the right-hand side members of  \eqref{eq:eqL1} and \eqref{eq:eqL2} to the region $D$. In a second step, we can also make the meromorphic continuation of $u(p,q)A(p)$ and   $v(p,q)B(q)$ to the region $D\cap \{K(p,q)=0\}$, which leads to the relation
\begin{equation}
\label{eq:BVP}
   u(p,q) A(p) + v(p,q) B(q) =0, \quad \forall (p,q)\in D\cap \{K(p,q)=0\}.
\end{equation}
We remember that $A(p)$ and $B(q)$ are sought to be analytic, respectively for~$\Re(p)\geq0$ and~$\Re(q)\geq0$. Then, after a  now standard process (see, e.g., \cite{FIM-2017}), we can make the meromorphic continuation of $A(p)$ (resp.\ $B(q)$) to $\mathbb{C}\setminus (-\infty,p_1]$ (resp.\ $\mathbb{C}\setminus (-\infty,q_1]$), remarking that equation \eqref{eq:BVP} remains valid on these extended domains for $(p,q)\in\{K(p,q)=0\}$. A detailed version of this extension procedure is given in Proposition~\ref{prop:appendix} of the Appendix. Hence, we get from \eqref{eq:BVP}, for $q\in [q_2,\infty)$,
\begin{equation*}
\begin{cases}
u(P_1(q),q) A(p) + v(P_1(q),q) B(q) =0, \\[0.2cm]
u(P_2(q),q) A(p) + v(P_2(q),q) B(q) =0.
\end{cases}
\end{equation*}
Now, using the continuity of $B(q)$ when $q$ traverses the cut~$[q_2,\infty)$, we obtain directly the announced BVP~\eqref{eq:statement_BVP}.
The proof of the theorem is concluded.
\end{proof}

To express the solution of the  BVP~\eqref{eq:statement_BVP}, we need to introduce the following functions.  First, for~$x\in\mathbb{C}\setminus (-\infty,-1]$ and $a\in \mathbb{R}$, let 
\begin{equation*}
T_a(x) \egaldef   \cos \left(a \arccos x \right)
=\frac{1}{2}
\Bigl(\bigl(x+\sqrt{x^2-1}\bigr)^a+\bigl(x-\sqrt{x^2-1}\bigr)^a\Bigr).
\end{equation*}
Then, we define the function $w$, analytic on $\mathbb{C}\setminus (-\infty,p_1]$, by
\begin{equation}
\label{def:w}
   w(p)\egaldef T_{\frac{\pi}{\beta}}\left(\frac{2p-(p_1 +p_2)}{p_2 -p_1}\right),
\end{equation}
where
\begin{equation}
\label{def:beta}
   \beta \egaldef \arccos \left( -\frac{\rho}{\sqrt{\sigma_1\sigma_2}} \right) \in (0,\pi).
\end{equation}
The angle $\beta$ is equal to $2\pi-\zeta$, with $\zeta$ as in the introduction. This function $w$  is referred to as a \emph{conformal gluing function}, see \cite{LIT-1977,FIM-2017}. This name is justified by the following lemma.
\smallskip
\begin{lem} 
The function $w$ in \eqref{def:w} satisfies the following properties:
  \begin{enumerate}[label={\rm(\arabic*)},ref={\rm(\arabic*)}] 
  \item It is analytic in an open domain containing $\H_{+}^o$, namely $\mathbb{C}\setminus  (-\infty,p_1]$.
\item It is bijective from $\H^o_+$ to $\mathbb{C}\setminus (-\infty, -1]$.
  \item It satisfies the boundary condition
    \begin{equation*}
      w(p)=w(\overline{p}), \quad \forall p\in\H.
    \end{equation*}
  \end{enumerate}
\end{lem}

The proof is analogous to the one found in \cite{BoElFrHaRa-21} or in \cite{FrRa-19}.
\medskip
\begin{thm}
\label{thm:main_result}
The function $A(p)$ in \eqref{eq:fourier} has the following integral explicit expression in the domain $\H_{+}^{o}$:
\begin{multline}
\label{eq:formula_solution}
   A(p)= \\\frac{\mu_1r_2-\mu_2}{1-r_1r_2} 
\left( \frac{w(0)-w(p_0)}{w(p)-w(p_0)} \right)^{-\chi}
\exp \left[ 
\int_{\H_\pm} \log \left( \frac{g(t)}{g(\bar{t})}\right) \left(\frac{w'(t)}{w(t)-w(p)}-\frac{w'(t)}{w(t)-w(0)}\right) \mathrm{d}t
\right],
\end{multline}
where
\begin{equation}
\label{eq:formula_index}
\chi=
\begin{cases}
0  & \text{if } u(P_i(q_2),q_2)\geqslant 0, \\[0.2cm]
-1  & \text{if } u(P_i(q_2),q_2)<0
\end{cases}
\end{equation}
is called the index,
and
\begin{equation}
\label{eq:formula_poles}
   p_0= \frac{2(\mu_2r_1 -\mu_1)}{\sigma_1+\sigma_2r_1^2-2\rho r_1}<0
\end{equation}
is the pole.
\end{thm}
\begin{proof}
The proof is similar to the one of  Theorem~1 in \cite{FrRa-19}. The main differences are that the BVP takes place on another branch of hyperbola, and the domain of the BVP is on the right of the hyperbola (not on the left as in the quarter plane), see Figure~\ref{fig:domain}. This leads to a formula that looks like the one of the quarter plane, but with several differences concerning for instance the pole and the index. This will be made more precise in Section~\ref{sec:diff}.

From \eqref{eq:formula_poles}, $p_0<0$, as follows from the next two inequalities:
\begin{itemize}
\item $\mu_2r_1 -\mu_1<0$, by \eqref{eq:CNS2_ergodic};
\item $\sigma_2r_1^2-2\rho r_1 +\sigma_1>0$, since, from the ellipticity condition \eqref{eq:elliptic}, this second degree  polynomial in $r_1$ is always positive.\qedhere
\end{itemize}
\end{proof}
When $\chi=0$, there is no pole inside the domain of the BVP. When $\chi=-1$, the ergodicity conditions imply $p_0<0$. This is consistent with the fact that $A(p)$ must be analytic for~$\Re(p) >0$. 

\section{Differences and similarities between $1/4$ and $3/4$-plane}
\label{sec:diff}
 
We intend to compare, as clearly as possible, the $1/4$ and $3/4$-plane situations, and we will frequently refer to Figures~\ref{fig:domain},~\ref{fig:ellipse} and \ref{fig:tableau}.

\subsubsection*{$\bullet$ Laplace and Fourier transforms:} While the Fourier transform of $\pi$ exists without any issue both in the quarter and three-quarter planes, there is a problem to define the Laplace transform in the three-quarter plane, due to the non-convex integration domain, which leads to the existence problem of $L(p,q)$. However, by a convenient continuation procedure, it is possible to define the Laplace transform in a neighborhood of $0$,  showing analytically that the stationary distribution decays exponentially in all directions. 

We can further remark that $A(p)$ and $B(q)$, the Laplace transforms~\eqref{eq:fourier} of the boundary invariant measures $\nu_1$ and $\nu_2$, are integrals defined on  different sets ($\mathbb{R}_-$ for the quarter plane, and $\mathbb{R}_+$ for the three-quarter plane).

\subsubsection*{$\bullet$ Kernel and reflections:} The algebraic formulas for the kernel $K$ in \eqref{eq:kernel}, the reflection polynomials $u,v$ in \eqref{eq:coef} and the functional equation \eqref{eq:eqfunc} are exactly the same ones in the two frameworks. On the other hand, the set of parameters for $u$ and $v$ are different (this is due to different ergodicity conditions in the quarter and three-quarter planes).

\subsubsection*{$\bullet$ Boundary value problem:} The boundary condition \eqref{eq:statement_BVP} has the same algebraic form in both cases; however, it is located on different branches of the same hyperbola, as represented on Figure~\ref{fig:domain}. This is due to the different initial convergence problem of $A$ and $B$.

\subsubsection*{$\bullet$ Solution of the above boundary value problems:} The formula \eqref{eq:formula_solution} for the solution has the same algebraic expression as its quadrant analogue in terms of the index, the function $w$, the pole $p_0$, etc. However, the values of the index $\chi$ in \eqref{eq:formula_index}, the pole \eqref{eq:formula_poles} and the function $w$ \eqref{def:w} are different. See Figures~\ref{fig:ellipse} and \ref{fig:tableau}. 

\medskip
These differences imply in turn that, for  given negative drift and covariance matrix, there might be no pole in the BVP domain $\H_{+}^{o}$, regardless of the reflection vectors. On the contrary, in the quadrant case, for given negative drift and covariance matrix, it is always possible to choose reflection vectors to ensure the existence of a pole in the BVP domain.

\begin{center}
\begin{figure}[h]
\begin{tabular}{|lc|c|c|}
\hline 
& & Three-quarter plane & Quarter plane \\ 
\hline 
Hyperbola & $\H$ & 
$
\begin{cases}
\H_p^+ & \text{if } \rho<0
\\
\H_p^- & \text{if } \rho>0
\end{cases}
$
 & 
$
\begin{cases}
\H_p^+ & \text{if } \rho>0
\\
\H_p^- & \text{if } \rho<0
\end{cases}
$
  \\ 
\hline 
BVP domain & $\H^{o}_\pm$ & $\H_{+}^{o}$ & $\H_{-}^{o}$
\\
\hline
Gluing function & $w$ & $T_{\frac{\pi}{\beta}}\left(+\frac{2p-(p_1 +p_2)}{p_2 -p_1}\right)$ & $T_{\frac{\pi}{\beta}}\left(-\frac{2p-(p_1 +p_2)}{p_2 -p_1}\right)$ \\ 
\hline 
Pole &$p_0$ & $\frac{2(\mu_2r_1 -\mu_1)}{\sigma_1+\sigma_2r_1^2-2\rho r_1}<0$ & $\frac{2(\mu_2r_1 -\mu_1)}{\sigma_1+\sigma_2r_1^2-2\rho r_1}>0$ \\ 
\hline 
Index & $\chi$ & $\begin{cases}
0  & \text{if } u(P_i(q_2),q_2)\geqslant 0
\\
-1  & \text{if } u(P_i(q_2),q_2)<0
\end{cases}$ & 
$
\begin{cases}
0  & \text{if } u(P_i(q_1),q_1)\leqslant 0
\\
-1  & \text{if } u(P_i(q_1),q_1)>0
\end{cases}
$
 \\ 
\hline 
\end{tabular} 
\caption{Differences between the formulas in the $1/4$-plane and in the $3/4$-plane.}
\label{fig:tableau}
\end{figure}
\end{center}

\subsubsection*{$\bullet$ Inverse Laplace transform:} In both cases, one can  invert the Laplace and Fourier transforms, using standard inversion formulas. For the three-quarter plane $S$, we consider a two-dimensional bilateral Laplace transform assuming that $\Pi$ is defined and equal to zero on the quarter plane complementary to $S$, i.e., $z_1\leq0$ and $z_2\leq0$. We have
\begin{equation*}
\begin{cases}
\DD \Pi(z_1,z_2)=\frac{1}{(2i\pi)^2}\int_{-i\infty}^{i\infty} \int_{-i\infty}^{i\infty} e^{-pz_1-qz_2} L(p,q) \mathrm{d}p \mathrm{d}q, \quad \forall (z_1,z_2)\in S,
\\[0.4cm]
\DD \nu_1(z_1)=\frac{1}{2i\pi}\int_{-i\infty}^{i\infty} e^{-pz_1} A(p) \mathrm{d}p, \quad \forall z_1<0,
\\[0.4cm]
\DD \nu_2(z_2)=\frac{1}{2i\pi}\int_{-i\infty}^{i\infty} e^{-qz_2} B(q) \mathrm{d}q, \quad \forall z_2<0.
\end{cases}
\end{equation*}

\appendix

\section{Analytic continuation of the Laplace transforms on the universal covering of the algebraic kernel curve}
\label{sec:continuation}

The goal of this appendix is to give more details about the analytic continuation of $A$ and $B$ done in the proof of Theorem~\ref{thm:BVP}. 
The proof, inspired from a standard analytic continuation process (see, e.g., \cite{FIM-2017}), uses the universal covering of the algebraic kernel curve. 
\begin{prop}
\label{prop:appendix}
The functions $A(p)$ (resp.\ $B(q)$) can be meromorphically continued to $\mathbb{C}\setminus (-\infty,p_1]$ 
(resp.\ $\mathbb{C}\setminus (\infty,q_1]$). Furthermore, we have 
\begin{equation*}
u(p,q) A(p)+v(p,q) B(q)=0,
\end{equation*}
for all $p\in\mathbb{C}\setminus (-\infty,p_1], q\in\mathbb{C}\setminus (-\infty,q_1]$, with
$\{K(p,q)=0\}$.
\end{prop}

\begin{proof}
The idea of the proof is to use the third function $F$ defined on $D$ to fill the gap between the initial definition domains $D_1$ and $D_2$. A similar approach based on the universal covering has been recently carried out for random walks in \cite{EP-22}.
We start from the two equations (proved  in (4.1) and (4.2) of~\cite{FaFrRa-21})
\begin{equation}
\label{eq:sys1}
{K(p,q)L_1(p,q)+ u(p,q) A(p)}+ F(p,q)
= 0,
\end{equation}
\hspace{2cm} \emph{in the region} $D_1 \cap D = \{\Re{(p)}\geq0 \} \cap \{\Re{(p+q)}\leq0\}$;
\begin{equation}\label{eq:sys2}
{K(p,q)L_2(p,q)+ v(p,q) B(q)} -  F(p,q)
= 0, 
\end{equation}
\hspace{2cm} \emph{in the region} $D_2 \cap D = \{\Re{(q)}\geq0 \} \cap \{\Re{(p+q)}\leq0\}$.

Note that  $A$ is initially defined on the domain $D_1$, $B$ on $D_2$, $F$ on $D$, $L_1$ on $D_1 \cap D$ and $L_2$ on $D_2 \cap D$.
In addition, the so-called kernel $K(p,q)$ is the same for the two functional equations~\eqref{eq:sys1} and~\eqref{eq:sys2}. 

Then,  when $K(p,q)=0$, we have
\begin{equation*}
\begin{cases}
u(p,q) A(p)+  F(p,q) = 0 & \text{in } D_1\cap D,
\\
v(p,q) B(q) -  F(p,q) = 0 & \text{in } D_2\cap D.
\end{cases} 
\end{equation*}
We shall  use the universal covering of the Riemann surface, which was also studied in detail in the appendix of \cite{BoElFrHaRa-21}. All the functions can be lifted on the universal covering of the Riemann surface $K(p,q)=0$. Since this surface has genus $0$, it can be uniformized by means of rational functions. Here we can take
\begin{equation*}
\label{eq:uniformization}
\left\{
  \begin{array}{l}
         p(s) =\displaystyle \frac{p_1+p_2}{2}+\frac{p_2-p_1}{4}\left(s+\frac{1}{s}\right),\smallskip\\
         q(s) = \displaystyle\frac{q_1+q_2}{2}+\frac{q_2-q_1}{4}\left(\frac{s}{e^{i\beta}}+\frac{e^{i\beta}}{s}\right) ,
  \end{array}\right.
\end{equation*}
where $\beta$ is defined in \eqref{def:beta}. We have
\begin{equation*}
\{(p,q)\in (\mathbb{C}\cup \{ \infty \})^2:  K(p,q)=0 \}=\{ (p(s),q(s)) : s\in\mathbb{C}\cup  \{ \infty \} \}.
\end{equation*}
Then, by means of the mapping $e^{i\omega}$, we lift the functions onto the universal covering, setting
\begin{equation*}
\hat{A}(\omega) \egaldef A(p(e^{i\omega})), \quad 
\hat{B}(\omega) \egaldef B(q(e^{i\omega}))\quad \text{and} \quad \hat{F}(\omega)\egaldef F(p(e^{i\omega}),q(e^{i\omega})).
\end{equation*}
Analogously, we define $\hat{u}(\omega)$ and $\hat{v}(\omega)$.
We then have 
\begin{equation}
\label{eq:coupleeqrev}
\begin{cases}
\hat u(\omega) \hat A(\omega)+ \hat F(\omega) = 0 & \text{ in } \hat D_1\cap \hat D,
\\
\hat v(\omega) \hat B(\omega) -  \hat F(\omega) = 0 & \text{ in } \hat D_2\cap  \hat D,
\end{cases} 
\end{equation}
where we consider the lifted domains, see Figure~\ref{fig:revetement}:
\begin{align*}
   \hat{D}_1 &= \{\omega \in\mathbb{C}: \Re (p(e^{i\omega})) >0 \text{ and }  \pi< \Re(\omega) <3\pi \} ,\\
   \hat{D}_2 &= \{\omega \in\mathbb{C}: \Re (q(e^{i\omega})) >0 \text{ and }  -\pi+\beta< \Re (\omega) <\pi+\beta \},\\
\hat{D}&= \{\omega \in\mathbb{C}: \Re (p(e^{i\omega})+q(e^{i\omega})) <0 \text{ and }  0< \Re (\omega) <2\pi \}.
\end{align*}
The precise shape of $\hat{D}_1$ and $\hat{D}_2$, and the curves bounding these domains have been studied in \cite{BoElFrHaRa-21}. The equations of the curves bounding $\hat{D}$ can be found in the same way, and, for $\omega=x+iy$, with $ 0< \Re(\omega) <2\pi$, we obtain that $\omega \in \hat D$ if and only if
\begin{equation*}
   \sqrt{\sigma_2} \bigl(\cos(\theta) + \cos(x) \cosh(y) \bigr) + \sqrt{\sigma_1}\bigl( \cos(\beta - \theta) + \cos(x - \beta) \cosh(y)\bigr) = 0.
\end{equation*}
See the green curves in Figure~\ref{fig:revetement}. 

The functions $\hat A_1,\hat B,\hat F$  are  initially defined, respectively, in 
$\hat D_1,\hat D_2,\hat D$. Then, thanks to \eqref{eq:coupleeqrev}, we can continue meromorphically~$\hat A$ and $\hat B$ to the domain $\hat D$, remarking  that $\hat D_1$ and $\hat D_2$ have an empty intersection, the gap between these two sets being filled by~$\hat D$.  This phenomenon is illustrated in Figure~\ref{fig:revetement}.

Now, it is possible to sum up the two equations~\eqref{eq:coupleeqrev}, which leads to
\begin{equation*}
\hat u(\omega) \hat A(\omega)+\hat v(\omega) \hat B(\omega)=0, \quad \forall \omega \in \hat D.
\end{equation*}
 Hence, $\hat A$ and $\hat B$ can be meromorphically continued to the domain~$\hat D\cup\hat D_1\cup\hat D_2$, and finally, by a standard procedure, to the whole of the complex plane.

\begin{figure}[hbtp]
\centering
\includegraphics[trim = 1cm 0.5cm 1cm 0.5cm, clip,scale=1.9]{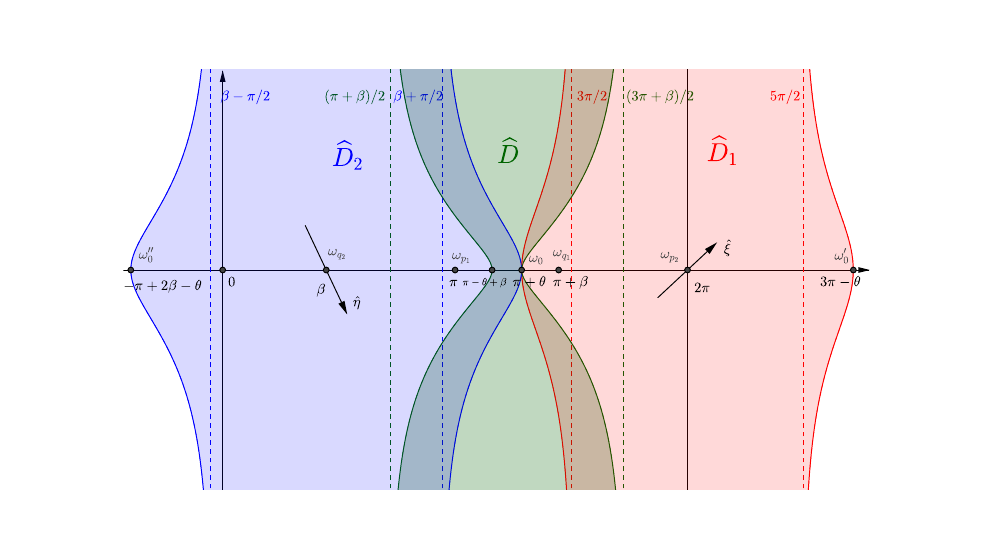}
\caption{Universal covering of the genus zero Riemann surface $\{(p,q): K(p,q)=0 \}$. In red, $\hat D_1= \{ \omega : \Re (p(\omega))>0 \}$ is the initial definition domain of $\hat A$ and in blue, $\hat D_2= \{ \omega : \Re (q(\omega))>0 \}$ is the initial definition domain of $\hat B$. The green domain $\hat D = \{ \omega : \Re (p(\omega)+q(\omega))<0 \}$ is the initial definition domain of the catalytic function $\hat F$. The domain $\hat D$ fills the gap between $\hat D_1$ and $\hat D_2$, thus allowing to continue the functions.}
\label{fig:revetement}
\end{figure}

Turning back to the initial functions, we can continue meromorphically $A(p)$ (resp.\ $B(q)$) to $\mathbb{C}\setminus (-\infty,p_1]$ 
(resp.\ $\mathbb{C}\setminus (\infty,q_1]$). We obtain 
\begin{equation*}
u(p,q) A(p)+v(p,q) B(q)=0 ,
\end{equation*}
for all $p\in \mathbb{C}\setminus (-\infty,p_1], q\in \mathbb{C}\setminus (-\infty,q_1]$ with 
$\{K(p,q)=0 \}$.
Note that one can reobtain the  fundamental ``Fourier'' functional equation~\eqref{eq:eqfunc}, just by putting
\begin{equation*}
L(p,q) \egaldef L_1(p,q)+L_2(p,q), \quad \forall (p,q)\in D,
\end{equation*}
which yields
\begin{equation*}
K(p,q)L(p,q)+u(p,q) A(p)+v(p,q) B(q)=0, \quad \forall (p,q)\in D.
\end{equation*}

\end{proof}

\bibliographystyle{abbrv}
\bibliography{biblio_FFR}

\end{document}